\documentclass{amsart}
\usepackage{amsfonts,amssymb,amscd,amsmath,latexsym,amsbsy,mathrsfs,amsxtra}
\usepackage{eucal}
\numberwithin{equation}{section}            
\theoremstyle{plain}
\newtheorem{thm}{Theorem}[section]

\newtheorem{prop}[thm]{Proposition}

\newtheorem{lem}[thm]{Lemma}
\newtheorem{cor}[thm]{Corollary}

\theoremstyle{remark}
\newtheorem{rema}[thm]{Remark}

\newcommand{\barB}{\overline{\phantom{m}}^B}
\newcommand{\barU}{\overline{\phantom{m}}}
\newcommand{\Bc}{{\mathcal{B}_{\bc}}}
\newcommand{\bc}{{\mathbf{c}}}

\newcommand{\bs}{{\bf{s}}}

\newcommand{\N}{{\mathbb N}}

\newcommand{\cB}{{\mathcal B}}
\newcommand{\cC}{{\mathcal C}}

\newcommand{\cF}{{\mathcal F}}

\newcommand{\cM}{{\mathcal M}}
\newcommand{\cO}{{\mathcal O}}

\newcommand{\cS}{{\mathcal S}}

\newcommand{\cV}{{\mathcal V}}

\newcommand{\field}{{\mathbb K}}

\newcommand{\gfrak}{{\mathfrak g}}

\newcommand{\hfrak}{{\mathfrak h}}
\newcommand{\height}{\mathrm{ht}}

\newcommand{\id}{{\mathrm{id}}}

\newcommand{\ot}{\otimes}

\newcommand{\sU}{\mathscr{U}}

\newcommand{\Uq}{U}

\newcommand{\uqg}{{U_q(\mathfrak{g})}}

\newcommand{\vep}{\varepsilon}
\newcommand{\Vect}{{\cV}ect}

\newcommand{\Xfrak}{{\mathfrak{X}}}

\newcommand{\Z}{{\mathbb Z}}

\begin{document}
\title[The bar involution for quantum symmetric pairs]
{The bar involution for quantum symmetric pairs -- hidden in plain sight}
\author[Stefan Kolb]{Stefan Kolb}
\dedicatory{To Jasper Stokman on his 50th birthday in gratitude and friendship}
\address{School of Mathematics, Statistics and Physics,
Newcastle University, Newcastle upon Tyne NE1 7RU, United Kingdom}
\email{stefan.kolb@newcastle.ac.uk}

%
\begin{abstract}
 We show that all quantum symmetric pair coideal subalgebras $\cB_\bc$ of Kac-Moody type have a bar involution for a suitable choice of parameters $\bc$. The proof relies on a generalized notion of quasi $K$-matrix. The proof does not involve an explicit presentation of $\cB_\bc$ in terms of generators and relations.
\end{abstract}
\maketitle

\section{Introduction}
Let $\gfrak$ be a symmetrizable Kac-Moody algebra and $\theta:\gfrak\rightarrow \gfrak$ an involutive Lie algebra automorphism of the second kind. The theory of quantum symmetric pairs, as developed by G.~Letzter in the finite setting \cite{a-Letzter99a}, studies certain coideal subalgebras $\cB_\bc$ of the quantized enveloping algebra $\Uq=\uqg$. The coideal subalgebras $\cB_\bc$ depend on a family of parameters $\bc$. They are quantum group analogs of $U(\gfrak^\theta)$ where $\gfrak^\theta=\{x\in \gfrak\,|\,\theta(x)=x\}$.

In March 2012, C.~Stroppel asked the author if there exists a bar involution for $\cB_\bc$. She conjectured that $\cB_\bc$ should allow an algebra automorphism $\barB:\cB_\bc\rightarrow \cB_\bc$ which is the identity on the Letzter generators $B_i\in \cB_\bc$ and maps $q$ to $q^{-1}$. Unfortunately, back in 2012, this author did not sufficiently appreciate the question.

This changed drastically with the appearance of the preprint versions of \cite{a-BaoWang18a} and \cite{a-ES18} in October 2013. Both these papers contained a bar involution for an explicit example of $\cB_\bc$ of type AIII. For this example, H.~Bao and W.~Wang used the bar involution to construct an element $\Xfrak$ (denoted by $\Upsilon$ in \cite{a-BaoWang18a}), now called the quasi $K$-matrix for $\cB_\bc$, which intertwines the Lusztig bar involution $\barU$ for $\Uq$ and the new bar involution $\barB$ for $\cB_\bc$ as follows
\begin{align}\label{eq:quasiK-intro}
   \overline{b}^B\Xfrak= \Xfrak \overline{b} \qquad \mbox{for all $b\in \cB_\bc$.}
\end{align}
Moreover, Bao and Wang used the quasi $K$-matrix to construct a natural family of $\cB_\bc$-module automorphisms $\mathcal{T}$ on the category of finite-dimensional  $\Uq$-modules. At this point it became clear that the family $\mathcal{T}$ was a hot candidate for a universal $K$-matrix which the author had been chasing previously in joint work with J.~Stokman \cite{a-KolbStok09}. Subsequently, this was worked out jointly with M.~Balagovi\'c \cite{a-BalaKolb19}.

The line of argument, first proposed in \cite{a-BaoWang18a} and then performed in \cite{a-BalaKolb15}, \cite{a-BalaKolb19} in a general Kac-Moody setting, was as follows. First the existence of a bar involution on $\cB_\bc$ was established via an explicit presentation of $\cB_\bc$ in terms of generators and relations. Then it was proved that the defining relation \eqref{eq:quasiK-intro} of the quasi $K$-matrix has an essentially unique solution.

Writing down $\cB_\bc$ in terms of generators and relations is a difficult problem in the general Kac-Moody setting. The quest for the bar involution for $\cB_\bc$ was often cited as an important motivation to find such a presentation of $\cB_\bc$, see for example the introductions of the papers \cite{a-CLW20}, \cite{a-dC19p}. However, G.~Lusztig asked the author in January 2019 if there is a construction of the bar involution for $\cB_\bc$ which does not rely on a presentation of $\cB_\bc$ in terms of generators and relations. Lusztig pointed out that the construction of the bar involution for the positive part $U^+$ in \cite[1.2.12]{b-Lusztig94} does not involve explicit knowledge of the quantum Serre relations for $\Uq$.

The quasi $K$-matrix is the quantum symmetric pair analog of the quasi $R$-matrix for $\Uq$. The quasi $R$-matrix for $\Uq$ has a description which does not involve Lusztig's bar involution, see \cite[Theorem 4.1.2 (b)]{b-Lusztig94}. Hence it is natural to ask for a bar involution free description of the quasi $K$-matrix. Such a construction was given in the quasi-split case (and in a more general setting of Nichols algebras of diagonal type) in \cite{a-KY20}. For suitable parameters $\bc$ there exists an element $\Xfrak$ satisfying a relation similar to \eqref{eq:quasiK-intro}, and the map
\begin{align*}
  b\mapsto \Xfrak \overline{b} \Xfrak^{-1} \qquad \mbox{for all $b\in \cB_\bc$}
\end{align*}
is a bar involution on $\cB_\bc$. The element $\Xfrak$ in \cite{a-KY20} is given by $\Xfrak=(\vep\ot \id)(\Theta^\theta)$ where $\vep$ is the counit and the element $\Theta^\theta$ is defined by \cite[(6.1)]{a-KY20}, see also \cite[Remark 6.3]{a-KY20}.

Along a slightly different line, A.~Appel and B.~Vlaar observed in \cite{a-AV20p} that the construction of the quasi $K$-matrix in \cite{a-BalaKolb19} can be reformulated without the bar involution $\barB$. Recall that $\cB_\bc$ is defined for a (generalized) Satake diagram $(X,\tau)$, where $X$ is a subset of the nodes $I$ of the Dynkin diagram of $\gfrak$, and $\tau:I\rightarrow I$ is an involutive diagram automorphism. The pair $(X,\tau)$ needs to satisfy certain compatibility conditions, see Section \ref{sec:genSatake}. The Letzter generators of $\cB_\bc$ are given by $B_i=F_i$ for $i\in X$ and $B_i=F_i-c_iT_{w_X}(E_{\tau(i)})K_i^{-1}$ for $i\in I\setminus X$, see Section \ref{sec:QSP}. With this notation, condition \eqref{eq:quasiK} is equivalent to the condition
\begin{align*}
  B_i \Xfrak = \Xfrak \overline{B_i'} \qquad \mbox{for all $i\in I$}
\end{align*}
where $B_i$ are the Letzter generators corresponding to the parameter family $\bc$, and $B_i'$ are Letzter generators corresponding to a different parameter family $\bc'$. One then obtains an algebra isomorphism $\Psi:\cB_{\bc'}\rightarrow \cB_{\bc}$ defined by $\Psi(b)=\Xfrak\overline{b}\Xfrak^{-1}$. This isomorphism satisfies $\Psi(B_i')=B_i$ and $\Psi(q)=q^{-1}$.

In \cite{a-AV20p} Appel and Vlaar are mostly concerned with the bar involution free formulation of the quasi $K$-matrix $\Xfrak$. The main point of the present note, which we still believe has novelty to it, is the observation that the above argument provides a proof of the existence of the bar involution for all symmetrizable Kac-Moody algebras and all (generalized) Satake diagrams if the parameters are chosen such that $\bc'=\bc$, see Corollary \ref{cor:bar-revisited}. Although Appel and Vlaar get very close to this observation, they stop just short of the formulation of the existence of the bar involution in \cite[Remark 7.2]{a-AV20p}. In particular, with the new perspective on the quasi $K$-matrix, the study of the quantum Serre relations for $\cB_\bc$ mentioned in \cite[Remark 7.2]{a-AV20p} is not necessary anymore.

This note is organised as follows. In Section \ref{sec:prelims} we fix notation for quantized enveloping algebras and recall the notion of a generalized Satake diagram as formulated in \cite{a-RV20} which generalizes the notion of an admissible pair in \cite[Definition 2.3]{a-Kolb14}. We promote the view that future papers on quantum symmetric pairs should be written in the setting of generalized Satake diagrams. In Section \ref{sec:quasiK} we recall the definition of $\cB_\bc$ and reprove the existence of the quasi $K$-matrix $\Xfrak$. While this result is contained in \cite[Theorem 7.4]{a-AV20p}, we feel that it is worthwhile to reproduce it in the precise notations and setting of \cite{a-BalaKolb19}. Section \ref{sec:bar-revisited} contains the existence of the isomorphism $\Psi$. As a consequence we obtain the main message of this note, namely the general, relation-free existence proof for the bar involution on $\cB_\bc$, see Corollary \ref{cor:bar-revisited} 
\section{Preliminaries}\label{sec:prelims}
\subsection{Quantum groups}
Let $\gfrak$ be a symmetrizable Kac-Moody algebra with generalized Cartan matrix $(a_{ij})_{i,j\in I}$ where $I$ is a finite set and Cartan subalgebra $\hfrak\subset \gfrak$. Let $\{d_i\,|\,i\in I\}$ be a set of relatively prime positive integers such that the matrix $(d_ia_{ij})$ is symmetric. Let $\Pi=\{\alpha_i\,|\,i\in I\}$ be the set of simple roots for $\gfrak$, let $\Phi$ be the root system, and let $Q=\Z\Pi$ be the root lattice. For $\beta=\sum_{i\in I}n_i\alpha_i\in Q$ we write $\height(\beta)=\sum_{i\in I}n_i$. Consider the symmetric bilinear form $(\cdot,\cdot):Q\times Q \rightarrow \Z$ defined by $(\alpha_i,\alpha_j)=d_ia_{ij}$ for all $i,j\in I$. We denote $Q^+=\N_0 \Pi$ where $\N_0=\{0,1,2, \dots\}$. Let $W$ be the Weyl group of $\gfrak$ which is generated by the simple reflections $s_i$ for $i\in I$.

Throughout this paper let $k$ be a field of characteristic zero and let $\field=k(q)$ be the field of rational functions in a variable $q$. We consider the quantized enveloping algebra $\Uq=\uqg$ as the $\field$-algebra with generators $E_i, F_i, K_i^{\pm 1}$ for $i\in I$ subject to the relations given in \cite[3.1.1]{b-Lusztig94}.  For $\beta=\sum_{i\in I}n_i\alpha_i\in Q$ we write $K_\beta=\prod_{i\in I}K_i^{n_i}$. Let $U^+$ and $U^-$ denote the subalgebras of $U$ generated by $\{E_i\,|\,i\in I\}$ and $\{F_i\,|\,i\in I\}$, respectively. For any $\mu\in Q^+$ we write $U^+_\mu=\mathrm{span}_\field\{E_{i_1}\dots E_{i_\ell}\,|\, \sum_{j=1}^\ell \alpha_{i_j}=\mu\}$ where $\mathrm{span}_\field$ denotes the $\field$-linear span. Moreover, define $U^-_{-\mu}=\omega(U^+_\mu)$ where $\omega:\Uq\rightarrow \Uq$ denotes the algebra automorphism given in \cite[3.1.3]{b-Lusztig94}. For any $i\in I$ let $T_i:\Uq\rightarrow \Uq$ denote the algebra automorphism denoted by $T_{i,1}''$ in \cite[37.1]{b-Lusztig94}. By \cite[39.4.3]{b-Lusztig94} the automorphisms $T_i$ satisfy braid relations.  Hence, for $w\in W$ with reduced expression $w=s_{i_1}\dots s_{i_\ell}$ there is a well-defined automorphism $T_w=T_{i_1}\dots T_{i_\ell}$ of $\Uq$.
For any $i\in I$ let $r_i, {}_ir: U^+\rightarrow U^+$ denote the Lusztig-Kashiwara skew-derivations which are defined uniquely by the relation
\begin{align}\label{eq:skew-def+}
  xF_i-F_ix=\frac{1}{q_i-q_i^{-1}}\big(r_i(x) K_i - K_i^{-1}{}_ir(x) \big) \qquad \mbox{for all $x\in U^+$}
\end{align}
where $q_i=q^{d_i}$, see \cite[3.1.6]{b-Lusztig94}. The maps $r_i$ and ${}_ir$ satisfy the relations $r_i(E_j)={}_ir(E_j)=\delta_{i,j}$ and the skew-derivation properties
\begin{align}\label{eq:skew-derivations}
  r_i(xy) = q^{(\alpha_i,\nu)}r_i(x)y + x r_i(y), \qquad 
  {}_ir(xy) = {}_ir(x) y + q^{(\alpha_i,\mu)} x\, {}_ir(y)
\end{align}
for all $x\in U^+_\mu$, $y\in U^+_\nu$.

Recall that the bar involution for $\Uq$ is the $k$-algebra automorphism $\barU:\Uq\rightarrow \Uq$, $x \mapsto \overline{x}$ defined by
\begin{align*}
  \overline{E_i}=E_i, \quad  \overline{F_i}=F_i, \quad  \overline{K_i}=K_i^{-1}, \quad \overline{q}=q^{-1}
\end{align*}  
for all $i\in I$.

To formulate the properties of the quasi $K$-matrix we need to work in an overalgebra $\mathscr{U}$ of $\Uq$. Let $\cO_{\mathrm{int}}$ denote the category of integrable $\Uq$-modules in category $\cO$ and let ${\cF}or :\cO_{\mathrm{int}}\rightarrow \Vect$ be the forgetful functor into the category of $\field$-vector spaces. We define $\mathscr{U}=\mathrm{End}({\cF}or)$ to be the set of natural transformations from ${\cF}or$ to itself. By construction, $\mathscr{U}$ is an algebra which contains $\Uq$ as a subalgebra. See \cite[Section 3]{a-BalaKolb19} for precise definitions and further details. Consider the elements of the product $\widehat{U^+} =\prod_{\mu\in Q^+}U^+_\mu$ as additive infinite sums of elements in $U^+$. Then $\widehat{U^+}$ is a subalgebra of $\mathscr{U}$, see \cite[Example 3.2]{a-BalaKolb19}.

\subsection{Generalized Satake diagrams}\label{sec:genSatake}
In \cite{a-Kolb14} quantum symmetric pairs were defined for involutive automorphisms of $\gfrak$ of the second kind. Such involutions are classified in terms of Satake diagrams $(X,\tau)$ where $X\subset I$ is a subset of finite type and $\tau:I\rightarrow I$ is an involutive diagram automorphism satisfying the compatibility conditions given in \cite[Definition 2.3]{a-Kolb14}. Associated to the finite type subset $X$ is a finite type root system $\Phi_X\subset \Phi$ with parabolic Weyl group $W_X\subset W$. Let $w_X\in W_X$ be the longest element and let $\rho_X^\vee\in \hfrak$ be the half sum of positive coroots of $\Phi_X$. To be a Satake diagram, the pair $(X,\tau)$ has to satisfy the two conditions
\begin{enumerate}
  \item $\tau(\alpha_j)=-w_X(\alpha_j)$ for all $j\in X$;\label{cond:1}
  \item If $i\in I\setminus X$ and $\tau(i)=i$ then $\alpha_i(\rho_X^\vee)\in \Z$.\label{cond:2}
\end{enumerate}
Condition \eqref{cond:1} implies that the map $\Theta=-w_X\circ \tau:Q\rightarrow Q$ is an involutive automorphism of the root lattice. 
It was observed by V.~Regelskis and B.~Vlaar that condition \eqref{cond:2} can be replaced by the weaker condition
\begin{enumerate}
  \item[(2')] If $\tau(i)=i$ and $a_{ji}=-1$ then $\Theta(\alpha_i)\neq -\alpha_i-\alpha_j$. \label{cond:2p}
\end{enumerate}
Following \cite{a-RV20}, we call pairs $(X,\tau)$ as above, satisfying conditions \eqref{cond:1}, (2') generalized Satake diagrams.
As explained in \cite[Section 4]{a-RV20} the construction of quantum symmetric pairs and the structure theory developed in \cite{a-Kolb14} remains valid for generalized Satake diagrams.

The diagram automorphism $\tau$ lifts to an Hopf-algebra automorphism of $\Uq$. As in \cite[3.1.3]{b-Lusztig94} let $\sigma:\Uq\rightarrow \Uq$ denote the $\field$-algebra antiautomorphism defined by
\begin{align*}
  \sigma(E_i)=E_i, \qquad \sigma(F_i)=F_i, \qquad \sigma(K_i)=K_i^{-1} \qquad\mbox{for all $i\in I$.}
\end{align*}  
The following proposition was conjectured in \cite[Conjecture 2.7]{a-BalaKolb15} and was proved in \cite[Proposition 2.5]{a-BalaKolb15} up to a sign. The sign was confirmed by H.~Bao and W.~Wang in \cite{a-BW18p} by a subtle argument involving canonical bases for $\Uq$. The arguments in \cite{a-BalaKolb15}, \cite{a-BW18p} do not involve condition (2) above, and hence also hold in the setting of generalized Satake diagrams, see also \cite[Lemma 7.14]{a-AV20p}.
\begin{prop}\label{prop:BK-conj}
   Let $(X,\tau)$ be a generalized Satake diagram for the symmetrizable Kac-Moody algebra $\gfrak$. Then the relation
  \begin{align*}
     \sigma\circ \tau(r_i(T_{w_X}(E_i)))=r_i(T_{w_X}(E_i))
  \end{align*}
  holds in $\Uq$ for all $i\in I\setminus X$.
\end{prop}  
\section{The quasi $K$-matrix, revisited}\label{sec:quasiK}
\subsection{Quantum symmetric pairs}\label{sec:QSP}
Let $(X,\tau)$ be an generalized Satake diagram. Set $\cM_X=\field\langle E_j, F_j, K_j^{\pm 1}\,|\,j\in X\rangle$ and set $U^0_\Theta=\field\langle K_\beta\,|\,\beta\in Q,\, \Theta(\beta)=\beta\rangle$. Define a set of parameters
\begin{align*}
  \cC=\{\bc=(c_i)_{i\in I\setminus X}\in (\field^\times)^{I\setminus X}\,|\, c_i=c_{\tau(i)} \mbox{ if $(\alpha_i,\Theta(\alpha_i))=0$}\}.
\end{align*}  
For any $\bc=(c_i)_{i\in I\setminus X}\in\cC$ we consider the subalgebra $\Bc$ of $\uqg$ generated by $\cM_X$, $U^0_\Theta$ and the elements
\begin{align}\label{eq:Bi-def}
  B_i=F_i - c_i T_{w_X}(E_\tau(i))K_i^{-1}\qquad \mbox{for all $i\in I\setminus X$.}
\end{align}
To unify notation, we set $c_i=0$ for $i\in X$ and extend \eqref{eq:Bi-def} by writing $B_i=F_i$ if $i\in X$.
Following \cite{a-Letzter99a}, \cite{a-Kolb14} we call $\Bc$ the quantum symmetric pair coideal subalgebra of $\uqg$ corresponding to the generalized Satake diagram $(X,\tau)$
\begin{rema}
  The parameters $c_i$ were denoted by $c_is(\tau(i))$ in \cite{a-BalaKolb19}. The additional parameters $s(\tau(i))$ appeared in the construction of involutive automorphisms of $\gfrak$ of the second kind corresponding to the Satake diagram $(X,\tau)$. For the construction of quantum symmetric pairs, however, it is advantageous to suppress the parameters $s(\tau(i))$ in the notation.
\end{rema}
\begin{rema}
  In \cite{a-Kolb14}, quantum symmetric pairs depend on a second family of parameters $\bs$ in a certain subset $\cS\subset \field^{I\setminus X}$. The corresponding coideal subalgebras are then denoted by $\cB_{\bc,\bs}$. By \cite[Theorem 7.1]{a-Kolb14} the algebra $\Bc$ is isomorphic to $\cB_{\bc,\bs}$ as an algebra for all $\bs$. As we only aim to establish the existence of a bar involution for $\cB_{\bc,\bs}$, it suffices to consider the case where $\bs=(0,0, \dots,0)$. Moreover, it was noted in \cite[3.5]{a-DobsonKolb19}, that for the construction of the quasi $K$-matrix there is no loss of generality in the restriction to the case $\bs=(0,0,\dots,0)$.
\end{rema}  
\subsection{The quasi $K$-matrix $\Xfrak$}
In \cite[Section 7]{a-AV20p} A.~Appel and V.~Vlaar gave a reformulation of the results of \cite[Section 6]{a-BalaKolb19} which does not rely on the existence of a bar involution for $\cB_\bc$. As notations and conventions in \cite{a-AV20p} somewhat differ from those of the present note, we feel it is beneficial to recall this reformulation in the present setting. We begin with a reformulation of \cite[Proposition 6.1]{a-BalaKolb19}. Let $2\rho_X$ be the sum of positive roots in $\Phi_X$.
\begin{prop}\label{prop:intertwiner}
  Let $\Xfrak=\sum_{\mu\in Q^+}\Xfrak_\mu\in \widehat{U^+}$ with $\Xfrak_\mu\in U^+_\mu$. The following are equivalent:
  \begin{enumerate}
    \item For all $i\in I$ the relation
      \begin{align*}
        B_i\Xfrak = \Xfrak\big(F_i- (-1)^{2\alpha_i(\rho_X^\vee)}q^{-(\alpha_i,\Theta(\alpha_i)-2\rho_X)}c_{\tau(i)} \overline{T_{w_X}(E_{\tau(i)})} K_i\big) 
      \end{align*}
      holds in $\mathscr{U}$.
    \item The element $\Xfrak$ satisfies the relations
      \begin{align}
        r_i(\Xfrak_\mu) &= (q_i-q_i^{-1}) (-1)^{2\alpha_i(\rho_X^\vee)} q^{-(\alpha_i,\Theta(\alpha_i)-2\rho_X)} c_{\tau(i)} \Xfrak_{\mu+\Theta(\alpha_i)-\alpha_i} \overline{T_{w_X}(E_{\tau(i)})} \label{eq:rec1},\\
        {}_ir(\Xfrak_\mu) &= (q_i-q_i^{-1}) q^{-(\Theta(\alpha_i),\alpha_i)} c_i T_{w_X}(E_{\tau(i)})  \Xfrak_{\mu+\Theta(\alpha_i)-\alpha_i} \label{eq:rec2}  
      \end{align}
      for all $i\in I$.
  \end{enumerate}
  If the equivalent conditions (1) and (2) hold, then the following also hold:
  \begin{enumerate}
    \item[(3)] For all $x\in U^0_\Theta \cM_X$ one has $x\Xfrak=\Xfrak x$.
    \item[(4)] For all $\mu\in Q^+$ such that $\Xfrak_\mu\neq 0$ one has $\Theta(\mu)=-\mu$.
  \end{enumerate}  
\end{prop}  
\begin{proof}
  The equivalence of (1) and (2) follows from the property \eqref{eq:skew-def+} of the skew derivations $r_i$ and ${}_ir$. Property (4) follows by induction over the height of $\mu$ and in turn implies property (3) just as in the proof of \cite[Proposition 6.1]{a-BalaKolb19}. 
\end{proof}
Next we need \cite[Proposition 6.3]{a-BalaKolb19}. This proposition is a general statement about solving recursive relations of the type \eqref{eq:rec1}, \eqref{eq:rec2} simultaneously, and hence holds in our setting. Let $\langle\, , \, \rangle: U^-\ot U^+\rightarrow \field$ be the nondegenerate pairing considered in \cite[Section 2.3]{a-BalaKolb19}.
\begin{prop}\label{prop:riir-conditions}{\upshape{(\cite[Proposition 6.3]{a-BalaKolb19})}}
  Let $\mu \in Q^+$ with $\height(\mu)\ge 2$ and fix elements $A_i,\, {}_iA\in U^+_{\mu-\alpha_i}$ for all $i\in I$. Then the following are equivalent:
  \begin{enumerate}
  \item There exists an element $\underline{X}\in U^+_\mu$ such that
    \begin{align}\label{eq:solveX}
       r_i(\underline{X})=A_i \quad \mbox{and} \quad  {}_ir(\underline{X})={}_iA \qquad \mbox{for all $i\in I$.}
    \end{align}  
  \item The elements  $A_i,\, {}_iA$ have the following two properties:\\
    (a) For all $i,j\in I$ one has
    \begin{align*}
       r_i({}_jA) = {}_jr(A_i).
    \end{align*}
    (b) For all $i\neq j\in I$ one has
    \begin{align}
  \label{eq:AiAj-rel2}    \frac{-1}{q_i-q_i^{-1}}\sum_{s=1}^{1-a_{ij}}&
      \begin{bmatrix} 1-a_{ij}\\s\end{bmatrix}_{q_i}(-1)^s \langle F_i^{1-a_{ij}-s} F_j F_i^{s-1}, A_j \rangle\\& - \frac{1}{q_j-q_j^{-1}}\langle F_i^{1-a_{ij}},A_j \rangle =0.\nonumber
    \end{align}  
  \end{enumerate}
  Moreover, if the system of equation \eqref{eq:solveX} has a solution $\underline{X}$ then this solution is uniquely determined.  
\end{prop}  
We now translate \cite[Section 6.4]{a-BalaKolb19} into our setting. This is the crucial step where we also need Proposition \ref{prop:BK-conj}. We hence give all the details. Fix $\mu\in Q^+$ and assume that a collection $(\Xfrak_{\mu'})_{\mu'<\mu}$ with $\Xfrak_{\mu'}\in U^+_{\mu'}$ and $\Xfrak_0=1$ has already been constructed, and that this collection satisfies the relations \eqref{eq:rec1},\eqref{eq:rec2} for all $\mu'<\mu$ and all $i\in I$. Define
\begin{align}
  A_i &= (q_i-q_i^{-1})(-1)^{2\alpha_i(\rho_X^\vee)} q^{-(\alpha_i,\Theta(\alpha_i)-2\rho_X)}c_{\tau(i)} \Xfrak_{\mu+\Theta(\alpha_i)-\alpha_i}  \overline{T_{w_X}(E_{\tau(i)})} \label{eq:Ai}\\
  {}_i A&=(q_i-q_i^{-1}) q^{-(\Theta(\alpha_i),\alpha_i)} c_i T_{w_X}(E_{\tau(i)}) \Xfrak_{\mu+\Theta(\alpha_i)-\alpha_i}. \label{eq:iA}
\end{align}
We have the following analog of \cite[Lemma 6.7]{a-BalaKolb19}.
\begin{lem}
 The relation $r_i({}_jA)={}_jr(A_i)$ holds for all $i,j\in I$.
\end{lem}
\begin{proof}
  Using the skew-derivation properties \eqref{eq:skew-derivations} and the assumptions on ${}_iA, A_i$ we calculate
  \begin{align*}
    r_i({}_jA) &= (q_j-q_j^{-1}) q^{-(\Theta(\alpha_j),\alpha_j)} c_j T_{w_X}(E_{\tau(j)})r_i(\Xfrak_{\mu+\Theta(\alpha_j)-\alpha_j}) \\
    &\qquad + (q_j-q_j^{-1}) q^{-(\Theta(\alpha_j),\alpha_j)} c_j q^{(\alpha_i,\mu+\Theta(\alpha_j)-\alpha_j)} r_i(T_{w_X}(E_{\tau(j)}))\Xfrak_{\mu+\Theta(\alpha_j)-\alpha_j}\\
    &=(q_j-q_j^{-1})(q_i-q_i^{-1}) (-1)^{2\alpha_i(\rho_X^\vee)} q^{-(\Theta(\alpha_j),\alpha_j)}   q^{-(\alpha_i,\Theta(\alpha_i)-2\rho_X)}c_j c_{\tau(i)}  \cdot\\
    &\phantom{=(q_j-q_j^{-1})}  \cdot T_{w_X}(E_{\tau(j)})\Xfrak_{\mu+\Theta(\alpha_j)-\alpha_j+\Theta(\alpha_i)-\alpha_i}\overline{T_{w_X}(E_{\tau(i)})} \\
    &\qquad + (q_j-q_j^{-1}) q^{-(\Theta(\alpha_j),\alpha_j)} c_j q^{(\alpha_i,\mu+\Theta(\alpha_j)-\alpha_j)} r_i(T_{w_X}(E_{\tau(j)}))\Xfrak_{\mu+\Theta(\alpha_j)-\alpha_j}\\
  \end{align*}
  and similarly
  \begin{align*}
    {}_jr(A_i) &=(q_i-q_i^{-1}) (-1)^{2\alpha_i(\rho_X^\vee)} q^{-(\alpha_i,\Theta(\alpha_i)-2\rho_X)} c_{\tau(i)}\, {}_jr(\Xfrak_{\mu+\Theta(\alpha_i)-\alpha_i})\overline{T_{w_X}(E_{\tau(i)})}\\
    &\qquad+(q_i-q_i^{-1})(-1)^{2\alpha_i(\rho_X^\vee)} q^{(\alpha_j,\mu+\Theta(\alpha_i)-\alpha_i)-(\alpha_i,\Theta(\alpha_i)-2\rho_X)}\times\\
    & \qquad \qquad \qquad\qquad \qquad \qquad \qquad\qquad \quad  \times c_{\tau(i)} \Xfrak_{\mu+\Theta(\alpha_i)-\alpha_i}\, {}_jr(\overline{T_{w_X}(E_{\tau(i)})})\\
    &=(q_i-q_i^{-1})(q_j-q_j^{-1}) (-1)^{2\alpha_i(\rho_X^\vee)} q^{-(\alpha_i,\Theta(\alpha_i)-2\rho_X)}q^{-(\Theta(\alpha_j),\alpha_j)} c_{\tau(i)} c_j \times \\
    &\phantom{=(q_j-q_j^{-1})}\qquad \qquad \qquad \times T_{w_X}(E_{\tau(j)}) \Xfrak_{\mu+\Theta(\alpha_i)-\alpha_i+\Theta(\alpha_j)-\alpha_j}\overline{T_{w_X}(E_{\tau(i)})}\\
    &\qquad +(q_i-q_i^{-1})(-1)^{2\alpha_i(\rho_X^\vee)}  q^{(\alpha_j,\mu+\Theta(\alpha_i)-\alpha_i)-(\alpha_i,\Theta(\alpha_i)-2\rho_X)}\times\\
    & \qquad\qquad \qquad\qquad \qquad \qquad \qquad\qquad \quad  \times c_{\tau(i)} \Xfrak_{\mu+\Theta(\alpha_i)-\alpha_i}\, {}_jr(\overline{T_{w_X}(E_{\tau(i)})}).
  \end{align*}
  Comparing the above two expressions we see that the relation $r_i({}_jA)={}_jr(A_i)$ is equivalent to the relation
  \begin{align}\label{eq:goal}
   (q_j-q_j^{-1})& q^{-(\Theta(\alpha_j),\alpha_j)} c_j q^{(\alpha_i,\mu+\Theta(\alpha_j)-\alpha_j)} r_i(T_{w_X}(E_{\tau(j)}))\Xfrak_{\mu+\Theta(\alpha_j)-\alpha_j}\\
    &= (q_i-q_i^{-1})  (-1)^{2\alpha_i(\rho_X^\vee)} q^{(\alpha_j,\mu+\Theta(\alpha_i)-\alpha_i)-(\alpha_i,\Theta(\alpha_i)-2\rho_X)}\times\nonumber\\
    & \qquad\qquad \qquad\qquad \qquad \qquad \times c_{\tau(i)} \Xfrak_{\mu+\Theta(\alpha_i)-\alpha_i}\, {}_jr(\overline{T_{w_X}(E_{\tau(i)})}). \nonumber
  \end{align}
  By \cite[1.2.14]{b-Lusztig94} we have ${}_jr(\overline{T_{w_X}(E_{\tau(i)})}){=} q^{(\alpha_j,w_X(\alpha_{\tau(i)})-\alpha_j)}\overline{r_j(T_{w_X}(E_{\tau(i)}))}$. Hence both sides of \eqref{eq:goal} vanish unless $\tau(i)=j\in I\setminus X$, and in this case \eqref{eq:goal} can be rewritten as
  \begin{align*}
    &q^{(\alpha_j,\Theta(\alpha_i)-\Theta(\alpha_j)-\alpha_i)+(\alpha_i,\mu)} r_i(T_{w_X}(E_i)) \Xfrak_{\mu+\Theta(\alpha_j)-\alpha_j}\\
    &= (-1)^{2\alpha_i(\rho_X^\vee)}
    q^{(\alpha_j,\mu - \alpha_i-\alpha_j) - (\alpha_i,\Theta(\alpha_i)-2\rho_X)} \Xfrak_{\mu+\Theta(\alpha_j)-\alpha_j} \overline{r_j(T_{w_X}(E_j))}
  \end{align*}
  From now on we assume $\tau(i)=j$ and hence $\Theta(\alpha_i)-\alpha_i=\Theta(\alpha_j)-\alpha_j$. Hence we may assume that $\Xfrak_{\mu+{\theta(\alpha_j)-\alpha_j}}\neq 0$ and in this case the above equation is equivalent to the equation
  \begin{align*}
     q^{(\alpha_i,\mu)} r_i(T_{w_X}(E_i)) = (-1)^{2\alpha_i(\rho_X^\vee)} q^{(\alpha_{\tau(i)},\mu-\alpha_i)-(\alpha_i,\Theta(\alpha_i)-2\rho_X)}\overline{r_{\tau(i)}(T_{w_X}(E_{\tau(i)}))}.
  \end{align*}
  By Proposition \ref{prop:intertwiner}.(4) we have $(\alpha_i-\alpha_{\tau(i)},\mu)=0$ and hence the above equation can be rewritten as
  \begin{align*}
    \overline{r_i(T_{w_X}(E_i))}= (-1)^{2\alpha_i(\rho_X^\vee)} q^{(\alpha_i,\alpha_i-w_X(\alpha_i)-2\rho_X)}r_{\tau(i)}(T_{w_X}(E_{\tau(i)})). 
  \end{align*}
  By \cite[Lemma 2.9]{a-BalaKolb15}, which also holds for generalized Satake diagrams, and by Proposition \ref{prop:BK-conj} the above equation does indeed hold. 
\end{proof}
Next we obtain an analog of \cite[Lemma 6.8]{a-BalaKolb19}. We include the proof which is simplified along the lines of \cite[Remark 6.9]{a-BalaKolb19}.
\begin{lem}
  For all $i\neq j\in I$ the elements $A_i, A_j$ given by \eqref{eq:Ai} satisfy the relation \eqref{eq:AiAj-rel2}.
\end{lem}
\begin{proof}
  As $A_i\in U^+_{\mu-\alpha_i}$ we may assume that $\mu=(1-a_{ij})\alpha_i+\alpha_j$ and $\Theta(\mu)=-\mu$ as otherwise both terms of \eqref{eq:AiAj-rel2} vanish. By \cite[Lemma 6.4]{a-BalaKolb19} it suffices to consider the following two cases:

  \noindent{\bf Case 1: $\Theta(\alpha_i)=-\alpha_j$ and $a_{ij}$=0.} In this case $\mu=\alpha_i+\alpha_j$ and $(\alpha_i,\alpha_k)=(\alpha_j,\alpha_k)=0$ for all $k\in X$, and hence
  \begin{align*}
     A_i=(q_i-q_i^{-1}) c_j E_j, \qquad A_j=(q_i-q_i^{-1}) c_i E_i.
  \end{align*}
  Hence the left hand side of \eqref{eq:AiAj-rel2} vanishes.

  \noindent{\bf Case 2: $\Theta(\alpha_i)=-\alpha_i$ and $\Theta(\alpha_j)=-\alpha_j$.} In this case induction on $\mu'$ for $\mu'<\mu$ shows that $\Xfrak_{\mu'}\neq 0$ implies that $\mu'\in \mathrm{span}_{\N_0}\{2\alpha_i, 2\alpha_j\}$,
  see \cite[Lemma 6.5]{a-BalaKolb19}. Hence for  $\mu=(1-a_{ij})\alpha_i+\alpha_j$ we have $\Xfrak_{\mu+\Theta(\alpha_i)-\alpha_i}=\Xfrak_{\mu-2\alpha_i}=0$ and similarly $\Xfrak_{\mu+\Theta(\alpha_j)-\alpha_j}=0$. Hence $A_i=A_j=0$ which implies \eqref{eq:AiAj-rel2} also in this case. 
\end{proof}  
With the above preparations the following theorem is proved just as in \cite[Theorem 6.10]{a-BalaKolb19} using Proposition \ref{prop:riir-conditions}.
\begin{thm}\label{thm:quasiK}
  There exists a uniquely determined element $\Xfrak=\sum_{\mu\in Q^+}\Xfrak_\mu\in \widehat{U^+}$ with  $\Xfrak_0=1$ and $\Xfrak_\mu\in U^+_\mu$ such that the equality
  \begin{align}\label{eq:quasiK}
    B_i \Xfrak = \Xfrak \big(F_i- (-1)^{2\alpha_i(\rho_X^\vee)}q^{-(\alpha_i,\Theta(\alpha_i)-2\rho_X)}c_{\tau(i)} \overline{T_{w_X}(E_{\tau(i)})} K_i\big) 
  \end{align}
  holds in $\sU$ for all $i\in I$. Moreover, the element $\Xfrak$ commutes with all elements of $U^0_\Theta\cM_X$.
\end{thm}
\begin{rema}
  We call the element $\Xfrak\in \sU$ from the above theorem the quasi $K$-matrix corresponding to the quantum symmetric pair coideal subalgebra $\Bc$.
The quasi $K$-matrix is invertible in $\widehat{U^+}=\prod_{\mu\in Q^+}U^+_\mu$ because $\Xfrak_0=1\neq 0$.
\end{rema}  
\section{The bar involution for quantum symmetric pairs, revisited}\label{sec:bar-revisited}
Recall that $\field=k(q)$. Following \cite[Section 7.3]{a-AV20p} for any $\bc=(c_i)_{i\in I\setminus X}\in \field^{I\setminus X}$ we define $\bc'=(c_i')_{i\in I\setminus X}\in \field^{I\setminus X}$ by
\begin{align*}
  c_i'=(-1)^{2\alpha_i(\rho_X^\vee)}q^{(\alpha_i,\Theta(\alpha_i)-2\rho_X)} \overline{c_{\tau(i)}}.
\end{align*}
Observe that $\bc\in \cC$ if and only if $\bc'\in \cC$. For $\bc\in \cC$ let $\cB_{\bc'}$ denote the quantum symmetric pair coideal subalgebra corresponding to the parameters $\bc'$. We denote the generators \eqref{eq:Bi-def} for $\cB_{\bc'}$ by $B_i'$ to distinguish them from the corresponding generators $B_i$ of $\cB_\bc$. We immediately obtain the following consequence of Theorem \ref{thm:quasiK}.
\begin{cor}\label{cor:bar}
  For any $\bc\in \cC$ there exists a $k$-algebra isomorphism $\Psi:\cB_{\bc'}\rightarrow \cB_{\bc}$ such that
  \begin{align*}
     \Psi|_{\cM_XU^0_\Theta}=\overline{\phantom{m}}|_{\cM_X U^0_\Theta} \qquad \mbox{and} \qquad \Psi(B_i')=B_i \quad \mbox{for all $i\in I\setminus X$.}
  \end{align*}
In particular, we have $\Psi(q)=q^{-1}$ and $\Psi(K_\beta)=K_{-\beta}$ for all $\beta\in Q^\Theta$.
\end{cor}
\begin{proof}
  With the above notation, Equation \eqref{eq:quasiK} can be rewritten as
  \begin{align}\label{eq:quasiK2}
    B_i\Xfrak = \Xfrak \overline{B_i'} \qquad \mbox{for all $i\in I\setminus X$}
  \end{align}
  Hence there is a well-defined $k$-algebra homomorphism $\Psi:\cB_{\bc'} \rightarrow \cB_\bc$ such that $\Psi(b)=\Xfrak \overline{b} \Xfrak^{-1}$ for all $b\in \cB_{\bc'}$. Equation \eqref{eq:quasiK2} implies that $\Psi(B_i')=B_i$ for all $i\in I\setminus X$. Property (3) in Proposition \ref{prop:intertwiner} implies that $\Psi|_{\cM_XU^0_\Theta}=\overline{\phantom{m}}|_{\cM_X U^0_\Theta}$.
\end{proof}
In the special case that $\bc=\bc'$ the above corollary provides the desired bar involution for $\cB_\bc$.
\begin{cor}\label{cor:bar-revisited}
  Assume that the parameters $\bc=(c_i)_{i\in I\setminus X}\in \field^{I\setminus X}$ satisfy the relation
  \begin{align}\label{eq:ci-condition}
     \overline{c_i}= (-1)^{2\alpha_i(\rho_X^\vee)}q^{-(\alpha_i,\Theta(\alpha_i)-2\rho_X)}c_{\tau(i)}
  \end{align}
  for all $i\in I\setminus X$. Then there exists a $k$-algebra automorphism
  \begin{align*}
     \barB:\Bc\rightarrow \Bc, \quad b\mapsto \overline{b}^B
  \end{align*}
  such that $\barB|_{U^0_\Theta\cM_X}=\overline{\phantom{m}}|_{U^0_\Theta\cM_X}$ and $\overline{B_i}^B=B_i$ for all $i\in I\setminus X$. In particular, the automorphism $\barB$ of $\Bc$ satisfies the relation $\overline{q}^B=q^{-1}$.
\end{cor}
\providecommand{\bysame}{\leavevmode\hbox to3em{\hrulefill}\thinspace}
\providecommand{\MR}{\relax\ifhmode\unskip\space\fi MR }
\providecommand{\MRhref}[2]{%
  \href{http://www.ams.org/mathscinet-getitem?mr=#1}{#2}
}
\providecommand{\href}[2]{#2}

\end{document}